\documentclass[a4paper,twoside,12pt]{article}
\usepackage{amssymb,amsmath,amsthm,amsfonts,latexsym}
\usepackage{graphicx}
\usepackage[pdftex,bookmarks,colorlinks=false]{hyperref}
\usepackage{caption,subcaption}
\usepackage[hmargin=1.2in,vmargin=1.2in]{geometry}
\usepackage[auth-sc-lg,affil-sl]{authblk}
\newtheorem{theorem}{Theorem}[section]
\newtheorem{corollary}[theorem] {Corollary}
\newtheorem{definition}[theorem]{Definition}

\newtheorem{property}[theorem]{Property}
\newtheorem{problem}[theorem]{Problem}
\newtheorem{proposition}[theorem]{Proposition}
\newtheorem{remark}[theorem]{Remark}

\def\ni{\noindent}
\def\N{\mathbb{N}_0}
\def\cP{\mathcal{P}}

\title{\sc A Study on Integer Additive Set-Graceful Graphs}
\author{N. K. Sudev}
\affil{\small Department of Mathematics\\ Vidya Academy of Science \& Technology \\ Thalakkottukara, Thrissur, India.\\ E-mail: sudevnk@gmail.com}

\author{K. A. Germina}
\affil{\small Department of Mathematics\\ University of Botswana \\Gaborone, Botswana.\\ E-mail: srgerminaka@gmail.com}

\date{}

\begin{document}
\maketitle

\begin{abstract}
A set-labeling of a graph $G$ is an injective set-valued function $f:V(G)\to \cP(X)$, where $X$ is a finite set and $\cP(X)$ be its power set. A set-indexer of $G$ is  a set-labeling such that the induced function $f^{\oplus}:E(G)\to \cP(X)-\{\emptyset\}$ defined by $f^{\oplus}(uv) = f(u){\oplus}f(v)$ for every $uv{\in} E(G)$ is also injective. Let $G$ be a graph and let $X$ be a non-empty set. A set-indexer $f:V(G)\to \cP(X)$  is called a set-graceful labeling of $G$ if $f^{\oplus}(E(G))=\cP(X)-\{\emptyset\}$. A graph $G$ which admits a set-graceful labeling is called a set-graceful graph. An integer additive set-labeling is an injective function $f:V(G)\to \cP(\N)$, $\N$ is the set of all non-negative integers and an integer additive set-indexer is an integer additive set-labeling such that the induced function $f^+:E(G) \to \cP(\N)$ defined by $f^+ (uv) = f(u)+ f(v)$ is also injective. In this paper, we introduce the notion of integer additive set-graceful labeling of graphs analogous to the set-graceful labeling of graphs and study certain properties and characteristics of the graphs which satisfy this type of set-labeling.
\end{abstract}

\vspace{0.2cm}


\section{Introduction}

For all  terms and definitions, not defined specifically in this paper, we refer to \cite{BM1}, \cite{BLS}, \cite{FH} and \cite{DBW}. For more about graph labeling, we refer to \cite{JAG}. Unless mentioned otherwise, all graphs considered here are simple, finite and have no isolated vertices. 

The researches on graph labeling problems commenced with the introduction of $\beta$-valuations of graphs in \cite{AR}. Analogous to the number valuations of graphs, the concepts of set-labelings and set-indexers of graphs are introduced in \cite{A1} as follows.

Let $X$ be a non-empty set and $\cP(X)$ be its power set. A \textit{set-labeling} of a graph $G$ with respect to the set $X$ is an injective set valued function $f:V(G)\to \cP(X)$ with the induced edge function $f^+(E(G))\to \cP(X)$ is defined by $f^+(uv)=f(u)\oplus f(v)$ for all $uv\in E(G)$, where $\cP(X)$ is the set of all subsets of $X$ and $\oplus$ is the symmetric difference of sets. A set-labeling of $G$ is said to be a \textit{set-indexer} of $G$ if the induced edge function $f^+$ is also injective. A graph which admits a set-labeling (or a set-indexer) is called a set-labeled (or set-indexed) graph. It is proved in \cite{A1} that every graph has a set-indexer.

A set-indexer $f:V(G)\to \cP(X)$  of a given graph $G$ is said to be a {\em set-graceful labeling} of $G$ if $f^{\oplus}(E(G))=\cP(X)-\{\emptyset\}$. A graph $G$ which admits a set-graceful labeling is called a {\em set-graceful graph}.

The sumset of two non-empty sets $A$ and $B$, denoted by $A+B$, is the set defined by $A+B=\{a+b:a\in A, b\in B\}$. If either $A$ or $B$ is countably infinite, then their sumset $A+B$ is also countably infinite. Hence all sets we consider here are non-empty finite sets. If $C=A+B$, where $A\neq \{0\}$ and $B\neq\{0\}$, then $A$ and $B$ are said to be the \textit{non-trivial summands} of the set $C$ and $C$ is said to be a \textit{non-trivial sumset} of $A$ and $B$. Using the concepts of sumsets, the notion of integer additive set-labeling of a given graph $G$ is defined as follows. 

\begin{definition}{\rm 
Let $\N$ be the set of all non-negative integers. Let $X\subseteq \N$ and $\cP(X)$ be its power set. An {\em integer additive set-labeling} (IASL, in short) of a graph $G$ is an injective function $f:V(G)\to \cP(X)$ whose associated function $f^+:E(G)\to \cP(X)$ is defined by $f^+(uv)=f(u)+f(v), uv\in E(G)$. A graph $G$ which admits an integer additive set-labeling is called an integer additive set-labeled graph. (IASL-graph).}
\end{definition} 

\begin{definition}{\rm 
\cite{GA}, \cite{GS1} An {\em integer additive set-labeling} $f$ is said to be an \textit{integer additive set-indexer} (IASI, in short) if the induced edge function $f^+:E(G) \to \cP(X)$ defined by $f^+ (uv) = f(u)+ f(v)$ is also injective.  A graph $G$ which admits an integer additive set-indexer is called an \textit{integer additive set-indexed graph}.}
\end{definition} 

By an element of a graph $G$, we mean a vertex or an edge of $G$. The cardinality of the set-label of an element of a graph $G$ is called the {\em set-indexing number} of that element. An IASL (or an IASI) is said to be a $k$-uniform IASL (or $k$-uniform IASI) if $|f^+(e)|=k ~ \forall ~ e\in E(G)$, where $k$ is a positive integer. The vertex set $V(G)$ is called {\em $l$-uniformly set-indexed}, if all the vertices of $G$ have the set-indexing number $l$.

With respect to an integer additive set-labeling (or integer additive set-indexer) of a graph $G$, the vertices of $G$ has non-empty set-labels and the set-labels of every edge of $G$ is the sumsets of the set-labels of its end vertices and hence no element of a given graph can have $\emptyset$ as its set-labeling. Hence, we need to consider only non-empty subsets of $X$ for set-labeling the vertices or edges of $G$.  Hence, all sets we mention in this paper are finite sets of non-negative integers. We denote the cardinality of a set $A$ by $|A|$. We denote, by $X$, the finite ground set of non-negative integers that is used for set-labeling the vertices or edges of $G$ and cardinality of $X$ by $n$. 

In this paper, analogous to the set-graceful labeling of graphs, we introduce the notion of an integer additive  set-graceful labeling of a given graph $G$ and study its properties.

\section{Integer Additive Set-Graceful Graphs}

Certain studies have been done on set-graceful graphs in \cite{A1}, \cite{A2}, \cite{AGPR} and \cite{AGKS}. Motivated by these studies we introduce the notion of a graceful type of integer additive set-labeling as follows. 

\begin{definition}{\rm
Let $G$ be a graph and let $X$ be a non-empty set of non-negative integers. An integer additive set-indexer $f:V(G)\to \cP(X)-\{\emptyset\}$  is said to be an {\em integer additive set-graceful labeling} (IASGL) or a {\em graceful integer additive set-indexer} of $G$ if  $f^{+}(E(G))=\cP(X)-\{\emptyset,\{0\}\}$. A graph $G$ which admits an integer additive set-graceful labeling is called an {\em integer additive set-graceful graph} (in short, IASG-graph).}
\end{definition}

As all graphs do not admit graceful IASIs in general, studies on the structural properties of IASG-graphs arouse much interest. The choice of set-labels of the vertices or edges of $G$ and the corresponding ground set is very important to define a graceful IASI for a given graph. 

A major property of integer additive set-graceful graphs is established as follows.

\begin{property}\label{P-IASGL0a}
If $f:V(G)\to \cP(X)-\{\emptyset\}$ is an integer additive set-graceful labeling on a given graph $G$, then $\{0\}$ must be a set-label of one vertex of $G$.
\end{property}
\begin{proof}
If possible, let $\{0\}$ is not the set-label of a vertex in $G$. Since $X$ is a non-empty subset of the set $\N$ of non-negative integers, it contains at least one element, say $x$, which is not the sum of any two elements in $X$. Hence, $\{x\}$ can not be the set-label of any edge of $G$. This is a contradiction to the hypothesis that $f$ is an integer additive set-graceful labeling. 
\end{proof}

Examining the above mentioned property of IASG-graphs, it can be understood that if an IASL $f:V(G)\to \cP(X)-\{\emptyset\}$ of a given graph $G$ is an integer additive set-graceful labeling on $G$, then the ground set $X$ must contain the element $0$. That is, only those sets containing the element $0$ can be chosen as the ground sets for defining a graceful IASI of a given graph. 

Another trivial but an important property of certain set-labels of vertices of an IASG-graph $G$ is as follows.

\begin{property}\label{O-IASGL1}
Let $f:V(G)\to \cP(X)-\{\emptyset\}$ be an integer additive set-graceful labeling on a given graph $G$. Then, the vertices of $G$, whose set-labels are not the non-trivial sumsets of any two subsets of $X$, must be adjacent to the vertex $v$ that has the set-label $\{0\}$.
\end{property}
\begin{proof}
Let $A_i\neq \emptyset$ be a subset of $X$ that is not a non-trivial sumset of any two subsets of $X$. But since $f$ is an IASGL of $G$, $A_i$ must be the set-label of some edge of $G$. This is possible only when $A_i$ is the set-label of some vertex $v_i$ that is adjacent to the vertex $v$ whose set-label is $\{0\}$.
\end{proof}

\ni Invoking Property \ref{O-IASGL1}, we have the following remarks.

\begin{remark}\label{O-IASGL1a}{\rm
Let $x_i~\in ~X$ be not the sum of any two elements in $X$. Since, $f$ is an integer additive set-graceful labeling, $\{x_i\}$ must be the set-label of one edge, say $e$ of $G$. This is possible only when one end vertex of $e$ has the set-label $\{0\}$ and the other end vertex has the set-label $\{x_i\}$.}
\end{remark}

\begin{remark}\label{C-IASGL1a}{\rm 
Let $f:V(G)\to \cP(X)-\{\emptyset\}$ be an integer additive set-graceful labeling on a given graph $G$ and let $x_1$ and $x_2$ be the minimal and second minimal non-zero elements of $X$. Then, by Remark \ref{O-IASGL1a}, the vertices of $G$ that have the set-labels $\{x_1\}$ and $\{x_2\}$, must be adjacent to the vertex $v$ that has the set-label $\{0\}$.}
\end{remark}

\begin{property}\label{O-IASGL2a}
Let $A_i$ and $A_j$ are two distinct subsets of the ground set $X$ and let $x_i$ and $x_j$ be the maximal elements of $A_i$ and $A_j$ respectively. Then, $A_i$ and $A_j$ are the set-labels of two adjacent vertices of an IASG-graph $G$ is that $x_i+x_j\le x_n$, the maximal element of $X$.
\end{property}
\begin{proof}
Let $v$ be a vertex of $G$ that has a set-label $A_i$ whose maximal element $x_i$. If $v$ is adjacent to a vertex, say $u$, with a set-label $A_j$ whose maximal element is $x_j$, then $f^+(uv)$ contains the element $x_i+x_j$. Therefore, $x_i+x_j~\in ~ X$. Hence, $x_i+x_j\leq x_n$.
\end{proof}

\ni In view of Property, we can observe the following remarks.

\begin{remark}\label{O-IASGL2}{\rm 
Let $f:V(G)\to \cP(X)-\{\emptyset\}$ be an integer additive set-graceful labeling on a given graph $G$ and let $x_n$ be the maximal element of $X$. If $A_i$ and  $A_j$ are set-labels of two adjacent vertices, then $A_i+A_j$ is the set-label for the edge. Hence for any $x_i\in A_i , x_j\in A_j$, $x_i+x_j\le x_n$ and hence, if one of the set is having the maximal element, then the other can not have a non-zero element. Hence, $x_n$ is an element of the set-label of a vertex $v$ of $G$ if $v$ is a pendant vertex that is adjacent to the vertex labeled by $\{0\}$. It can also be noted that if $G$ is a graph without pendant vertices, then no vertex of $G$ can have a set-label consisting of the maximal element of the ground set $X$.}
\end{remark}

The following results establish the relation between the size of an IASG-graph and the cardinality of its ground set.

\begin{remark}\label{R-IASGL3}{\rm 
Let $f$ be an integer additive set-graceful labeling defined on $G$. Then, $f^+(E)=\cP(X)-\{\emptyset,\{0\}\}$. Therefore, $|E(G)|=| \cP(X)|-2 = 2^{|X|}-2 =2(2^{|X|-1}-1)$. That is, $G$ has even number of edges.}  
\end{remark}

\begin{remark}\label{R-IASGL4}{\rm 
Let $G$ be an IASG-graph, with an integer additive set-graceful labeling $f$. By Remark \ref{R-IASGL3}, $|E(G)|=2^{|X|}-2$. Therefore, the cardinality of the ground set $X$ is $|X|=\log_2[|E(G)|+2]$.}
\end{remark}

The conditions for certain graphs and graph classes to admit a integer additive set-graceful labeling are established in following discussions.

\begin{theorem}\label{T-IASGL5}
A star graph $K_{1,m}$ admits an integer additive set-graceful labeling if and only if $m=2^n-2$ for any integer $n>1$.
\end{theorem}
\begin{proof}
Let $v$ be the vertex of degree $d(v)>1$. Let $m=2^n-2$ and $\{v_1,v_2,\ldots,v_m\}$, be the vertices in $K_{1,m}$ which are adjacent to $v$. Let $X$ be a set of non-negative integers containing $0$. 

First, assume that $K_{1,m}$ admits an integer additive set-graceful labeling, say $f$. Then, by Remark \ref{R-IASGL3}, $|E(G)|=m=2^{|X|}-2$. Therefore, $m=2^n-2$, where $n=|X|>1$.

Conversely, assume that $m=2^n-2$ for some integer $n>1$. Label the vertex $v$ by the set $\{0\}$ and label the remaining $m$ vertices of $K_{1,m}$ by the remaining $m$ distinct non-empty subsets of $X$. Clearly, this labeling is an integer additive set-graceful labeling for $K_{1,m}$. This completes the proof.
\end{proof}

The following theorem checks whether a tree can be an IASG-graph. 

\begin{proposition}\label{P-IASGL6}
If a tree on $m$ vertices admits an integer additive set-graceful labeling, then $1+m=2^n$, for some positive integer $n>1$.
\end{proposition}
\begin{proof}
Let $G$ be a tree on $m$ vertices. Then, $|E(G)|=m-1$. Assume that $G$ admits an integer additive set-graceful labeling, say $f$. Then, by Remark \ref{R-IASGL3}, for a ground set $X$ of cardinality $n$, then $m-1=2^{|X|}-2$. Hence, $m+1=2^{|X|}$.
\end{proof}

\begin{corollary}\label{C-IASGL6a}
Let $G$ be a tree on $m$ vertices. For a ground set $X$, let $f:V(G)\to \cP(X)$ be be an integer additive set-graceful labeling on $G$. Then, $|X|=\log_2(m+1)$.
\end{corollary}
\begin{proof}
Let $G$ be a tree which admits an IASGL. Then, by Theorem \ref{P-IASGL6}, we have $m+1 = 2^{|X|}$, where $X$ is the ground set for labeling the vertices and edges of graphs. Hence, $|X| = \log_2(m+1).$
\end{proof}

\begin{theorem}\label{P-IASGL6b}
A tree $G$ is an IASG-graph if and only if it is a star $K_{1,\, 2^n-2}$, for some positive integer $n$.
\end{theorem}
\begin{proof}
If $G=K_{1,2^n-2}$, then by Theorem \ref{T-IASGL5}, $G$ admits an integer additive set-graceful labeling. Conversely, assume that the tree $G$ on $m$ vertices admits an integer additive set-graceful labeling, say $f$ with respect to a ground set $X$ of cardinality $n$. Therefore, all the $2^n-1$ non-empty subsets of $X$ are required for labeling the vertices of $X$.  Also, note that  $\{0\}$ can not be a set-label of any edge of $G$. Hence, all the remaining $2^n-2$ non-empty subsets of $X$ are required for the labeling the edges of $G$. 

It is to be noted that the set-labels containing $0$ are either the the sumsets of some other set-labels containing $0$ or not a sumset of any subsets of $X$. Let $0\in A_i\subseteq X$ be the set-label of an edge,say $e$, of $G$. Then, if $A_i$ is not a sumset of subsets of $X$, by Property \ref{O-IASGL1}, then $A_i$ must be the set-label of a vertex, say $u$, that is adjacent to the vertex $v$ having set-label $\{0\}$. Assume that $A_i$ is the sum set of two sets $A_r$ and $A_s$. If $e=v_rv_s$, where $A_r$ and $A_s$ are respectively the set-labels of $v_r$ and $v_s$, then $e$ will be an edge of $G$ which is in a of cycle of $G$, a contradiction to the fact that $G$ is a tree. Therefore, vertices, whose set-labels containing $0$, must be adjacent to the vertex $v$ which has the set-label $\{0\}$. 

Also, note that, by Remark \ref{O-IASGL2}, the vertices, whose set-labels containing the maximal element $x_n$ of $X$ must also be adjacent to the vertex labeled by $\{0\}$.  

Let $X_i$ be a subset of $X$ which contains either $0$ or $x_n$ and let $v_i$ be the vertex of $G$ that has the set-label $X_i$. Then, the set-label of the edge $vv_i$ is also $X_i$. Let $X_j$ be a subset of $X$ that contains neither $0$ nor $x_n$ and is the set-label of an edge $e$ of $G$. Then, if $X_j$ is not a sumset in $\cP(X)$, then $X_j$ must be the set-label of vertex which is adjacent to the vertex $v$ having the set-label $\{0\}$. If $X_j$ is the sumset of two subsets $X_r$ and $X_s$ and $e=v_rv_s$, where $X_r$ and $X_s$ are respectively the set-labels of $v_r$ and $v_s$, then as explained above, the edge $e$ will be in a cycle of $G$, a contradiction to the fact that $G$ is a tree. Therefore, $X_j$ must be the set-label of vertex which is adjacent to the vertex $v$ having the set-label $\{0\}$. 

Hence, all vertices of $G$ having non-empty subsets of $X$, other than $\{0\}$, as the set-labels must be adjacent to the vertex $v$ having set-label $\{0\}$. Hence, $G$ is a star graph $deg(v)=2^n-2$.
\end{proof}

We now check the admissibility of integer additive set-graceful labeling by path graphs and cycle graphs.

\begin{corollary}\label{C-IASGL6p}
For a positive integer $m>2$, the path $P_m$ does not admit an integer additive set-graceful labeling.
\end{corollary}
\begin{proof}
Every path is a tree and no path other than $P_2$ is a star graph. Hence, by Theorem \ref{P-IASGL6b}, $P_m,~ m>2$ is not an IASG-graph.
\end{proof}

\begin{proposition}\label{P-IASGL6c}
For any positive integer $m>3$, the cycle $C_m$ does not admit an integer additive set-graceful labeling.
\end{proposition}
\begin{proof}
Let $X$ be a ground set with $n$ elements. Since $C_m$ has $m$ edges, by Remark \ref{R-IASGL3}, 
\begin{equation}
m=2^n-2 \label{myequn1}
\end{equation}
Since $C_m$ has no pendant vertices, by Proposition \ref{O-IASGL2}, the maximal element, say $x_n$, will not be an element of any set-label of the vertices of $C_m$. Therefore, only $2^{n-1}-1$ non-empty subsets of $X$ are available for labeling the vertices of $C_m$. Hence,
\begin{equation}
m\le 2^{n-1}-1 \label{myequn2}
\end{equation}

Clearly, Equation \ref{myequn1} and Equation \ref{myequn2} do not hold simultaneously. Hence, $C_m$ does not admit an integer additive set-graceful labeling.
\end{proof}

An interesting question we need to address here is whether complete graphs admit integer additive set-graceful labeling.  We investigate the conditions for a complete graph to admit an integer additive set-graceful labeling and based on these conditions check whether the complete graphs are IASG-graphs. 

\begin{theorem}\label{T-IASGL7}
A complete graph $K_m$, does not admit an integer additive set-graceful labeling.
\end{theorem}
\begin{proof}
Since $K_2$ has only one edge and $K_3$ has three edges, by Remark \ref{R-IASGL3}, $K_2$ and $K_3$ do not have an integer additive set-graceful labeling. Hence, we need to consider the complete graphs on more than three vertices.

Assume that a complete graph $K_m,m>3$ admits an integer additive set-graceful labeling. Then, by Remark \ref{R-IASGL3}, $|E(G)|=2^{|X|}-2=\frac{m(m-1)}{2}$. That is, $2^{|X|-1}-1=\frac{m(m-1)}{4}$.
Since $|X|>1$, $2^{|X|-1}-1$ is a positive integer. Hence, $m(m-1)$ is a multiple of $4$. This is possible only when either $m$ or $(m-1)$ is a multiple of $4$. 

Since $|X|>1$, $2^{|X|-1}-1$ is a positive odd integer. Hence, for an odd integer $k$, either $m=4k$ or $m-1=4k$. Therefore, $2^{|X|-1}-1 = \frac{4k(4k-1)}{4}=k(4k-1)$ or $2^{|X|-1}-1=\frac{4k(4k-1)}{4}=k(4k+1)$. That is, $2^{|X|-1} =1+k(4k\pm 1)$. That is, if a complete graph $K_m$ admits an integer additive set-graceful labeling, then there exist an integral solution for the equation 
\begin{equation}
4k^2\pm k +1=2^n \label{myequn3}
\end{equation} 
where $k$ is an odd non-negative integer and $n>3$ be a positive integer.

The equation \eqref{myequn3} can be written as a quadratic equation as follows.

\begin{equation}
4k^2\pm k +(1-2^n)=0 \label{myequn4}
\end{equation}

\ni The value of $k$ is obtained from \eqref{myequn4} as $k=\frac{\pm 1\pm \sqrt{1-16(1-2^n)}}{8}=\frac{\pm 1\pm \sqrt{2^{n+4}-15}}{8}$, which can not be a non-negative integer for the values $n>3$. Hence, $K_m$ does not admit an integer additive set-graceful labeling.
\end{proof}

In this context, we need to find the conditions required for the given graphs to admit IASGLs. The structural properties of IASGL graphs are discussed in the following results. 

\begin{proposition}\label{P-IASGL1b}
Let $G$ be an IASG-graph. Then, the minimum number of vertices of $G$ that are adjacent to the vertex having the set-label $\{0\}$ is the number of subsets of $X$ which are not the non-trivial sumsets of any subsets of $X$.
\end{proposition}
\begin{proof}
Let $G$ admits an integer additive set-graceful labeling $f$. Then, $f^+(E(G))=\cP(X)-\{\emptyset,\{0\}\}$. Let $X_i$ be a non-empty subset of $X$, which is not a non-trivial sumset of any other subsets of $X$. Since $G$ is an IASG-graph, $X_i$ should be the set-label of some edge of $G$. Since $X_i$ is the set-label of edges of $G$ and is not a non-trivial sumset any two subsets of $G$, this is possible only when $X_i$ is the set-label of some vertex of $G$ which is adjacent to the vertex $v$ whose set-label is $\{0\}$. Therefore, the minimum number of vertices adjacent to $v$ is the number of subsets of $X$ which are not the sumsets of any two subsets of $X$. 
\end{proof}

Another important structural property of an IASG-graph is established in the following theorem.

\begin{theorem}\label{T-IASGL2b}
If a graph $G$ admits an IASGL $f$ with respect to a finite ground set $X$, then the vertices of $G$, which have the set-labels which are not non-trivial summands of any subset of $X$, are the pendant vertices of $G$.
\end{theorem}
\begin{proof} 
Let $f$ be an IASGL defined on a given graph $G$. Then, every subset of $X$, other than $\emptyset$ and $\{0\}$, must be the set-label of some edges of $G$. Let $X_i$ be not a non-trivial summand of any subset of $X$. Then, the vertex $v_i$ with set-label $X_i$ can not be adjacent to any other vertex $v_j$ with set-label $X_j$, where $X_j\neq \{0\}$ as the set-label of the edge $v_iv_j$ is $X_i+X_j$ which is not a subset of $X$. Hence, $v_i$ can be adjacent only to the vertex $v$ having the set-label $\{0\}$.  
\end{proof}

\noindent Invoking the above theorems, we can establish the following result. 

\begin{theorem}\label{T-IASGL3a}
If $G$ is an IASG-graph, then at least $k$ pendant vertices must be adjacent to a single vertex of $G$, where $k$ is the number of subsets of $X$ which are neither the non-trivial sumsets of subsets of $X$ nor the non-trivial summands of any subset of $X$. 
\end{theorem}
\begin{proof} 
Let $f$ be an IASGL defined on a given graph $G$. Then, every subset of $X$, other than $\emptyset$ and $\{0\}$, must be the set-label of some edges of $G$. If $X_i$ is not a sumset of any two subsets of $X$, then by Proposition \ref{P-IASGL1b}, the vertex $v_i$of $G$ with the set-label $X_i$ will be adjacent to the vertex $v$ with the set-label $\{0\}$. If $X_i$ is not a non-trivial summand of any subset of $X$, then by \ref{T-IASGL2b}, the vertex $v_i$ will be a pendant vertex. Therefore, the minimum number of pendant vertex required for a graph $G$ to admit an IASGL is the number of subsets of $X$ which are neither the non-trivial sumset of any two subsets of $X$ nor the non-trivial summands of any subset of $X$.
\end{proof}

The following result is an immediate consequence of the above theorems.

\begin{theorem}
Let $G$ be an IASG-graph which admits an IASGL $f$ with respect to a finite non-empty set $X$. Then, $G$ must have at least $|X|-1$ pendant vertices.
\end{theorem}
\begin{proof}
Let A graph $G$ admits an IASL, say $f$ with respect to a ground set $X=\{0,x_1,x_2,\ldots,x_n\}$. By Theorem \ref{T-IASGL3a}, the number of pendant vertices is equal to the number of subsets of $X$ which are neither the non-trivial sumsets nor the non-trivial summands of any subsets of $X$. Clearly, the set $\{0, x_n\}$ is neither a sumset of any two subsets in $X$ nor a non-trivial summand of any set in $X$. By Property \ref{O-IASGL2}, the vertex of $G$ with set-label $\{0, x_n\}$ can be adjacent to a single vertex that has the set-label $\{0\}$. Now, consider the three element sets of the form $X_i=\{0,x_i,x_n\};~ 1\le i<n$. If possible let, the set $X_i$ be the sumset of two subsets say $A$ and $B$ of $X$. Then, $A$ and $B$ can have at most two elements. Since $0\in X_i$, $0$ must belong to both $A$ and $B$. Hence, let $A=\{0,a\}$ and $\{0,b\}$. Then, $X_i=A+B=\{0,a,b,a+b\}$ which is possible only when $a=b$ and hence $A=B$ and it contradicts the injectivity of $f$. Therefore, $X_i$ is not a sumset of any other subsets of $X$. Since $x_n\in X_i$, by Property \ref{O-IASGL2}, it can not be a non-trivial summand of any subset of $X$. Therefore, $X_i$ can be the set-label of a pendant vertex of $G$ only. The number of three element subsets of $X$ of this kind is $n-1$. It is to be noted that be that a subset $\{0,x_i,x_j,x_n\}$ can be a sumset of two subsets $\{0,x_i\}$ and $\{0,x_j\}$ of $X$ if $x_n=x_i+x_j$. This property holds for all subsets of $X$ with cardinality greater than $3$. Hence, the minimum number of subsets of $X$ which are neither the non-trivial sumsets of any two subsets of $X$ nor the non-trivial summands of any other subsets of $X$ is $n=|X|-1$. Hence, by Theorem \ref{T-IASGL3a}, the minimum number of pendant vertices of an IASG-graph is $|X|-1$.
\end{proof}

An interesting question that arises in this context is about the existence of a graph corresponding to a given ground set $X$ such that the function $f:V(G)\to \cP(X)$ is a graceful IASI on $G$. Hence we introduce the following notion.

\begin{definition}{\rm 
Let $X$ be a non-empty finite set of non-negative integers. A graph $G$ which admits a graceful IASI with respect to the set $X$ is said to be a \textit{graceful graph-realisation} of the set $X$ with respect to the IASL $f$.}
\end{definition}

It can be noted that a star graph $K_{1,2^{|X|-2}}$ admits an IASGL. The question whether there is a non-bipartite graph that admits an IASGL with respect to a given ground set $X$ is addressed in the following theorem.

\begin{theorem}\label{T-IASGL6a}
Let $X$ be a non-empty finite set of non-negative integers containing the element $0$. Then, there exists a non-bipartite graceful graph-realisation $G$.
\end{theorem}
\begin{proof}
Let $X$ be a finite non-empty set of non-negative integers containing $0$. Let $\mathcal{A}$ be the collection of all subsets of $X$ which are not the non-trivial sumsets of any two subsets of $X$ and let $\mathcal{B}$ be the collection of subsets of $X$ which are not the non-trivial summands of any subsets of $X$. We need to construct an IASG-graph $G$ with respect to $X$. For this, first take a vertex, say $v_0$ and label this vertex by $\{0\}$. Now, mark the vertices $v_1,v_2,\ldots v_r$, where $r=|\mathcal{A}\cup \mathcal{B}|$ and label these vertices in an injective manner by the sets in $\mathcal{A}\cup \mathcal{B}$. In view of Proposition \ref{P-IASGL1b} and Theorem \ref{T-IASGL2b}, draw the edges from each of these vertices to the vertex $v_0$. Next, mark new vertices $v_{r+1},v_{r+2},\ldots, v_{r+l}=v_n$, where $l=|(\mathcal{A}\cup \mathcal{B})^c|$ and label these vertices in an injective manner by the sets in $(\mathcal{A}\cup \mathcal{B})^c$. Draw edges between the vertices $v_i$ and $v_j$, if $f(v_i)+f(v_j)\subseteq X$, for all $ 0 \le i,j \le |V|$. Clearly, this labeling is an IASGL for the graph $G$ constructed here.  
\end{proof}

Invoking all the above results, we can summarise a necessary and sufficient condition for a graph $G$ to admit a graceful IASI with respect to a given ground set $X$.

\begin{theorem}\label{T-IASGL6}
Let $X$ be a non-empty finite set of non-negative integers. Then, a graph $G$ admits a graceful IASI if and only if the following conditions hold.
\begin{enumerate}\itemsep0mm
\item[(a)] $0\in X$ and $\{0\}$ be a set-label of some vertex, say $v$, of $G$
\item[(b)] the minimum number of pendant vertices in $G$ is the number of subsets of $X$ which are not the non-trivial summands of any subsets of $X$.
\item[(c)] the minimum degree of the vertex $v$ is equal to the number of subsets of $X$ which are not the sumsets of any two subsets of $X$ and not non-trivial summands of any other subsets of $X$.
\item[(d)] the minimum number of pendant vertices that are adjacent to a given vertex of $G$ is the number of subsets of $X$ which are neither the non-trivial sumsets of any two subsets of $X$ nor the non-trivial summands of any subsets of $X$. 
\end{enumerate}
\end{theorem}
\begin{proof}
The necessary part of the theorem follows together from Property \ref{O-IASGL1}, Proposition \ref{P-IASGL1b}, Theorem \ref{T-IASGL2b} and \ref{T-IASGL3a} and the converse of the theorem follows from Theorem \ref{T-IASGL6a}. 
\end{proof}

\section{Conclusion}

In this paper, we have discussed the concepts and properties of integer additive set-graceful graphs analogous to those of set-graceful graphs and have done a characterisation based on this labeling. 

We note that the admissibility of integer additive set-indexers by the given graphs depends also upon the nature of elements in $X$. A graph may admit an IASGL for some ground sets and may not admit an IASGL for some other ground sets. Hence, choosing a ground set $X$ is very important in the process of checking whether a given graph is an IASG-graph.

Certain problems in this area are still open. Some of the areas which seem to be promising for further studies are listed below.

\begin{problem}{\rm
Characterise different graph classes which admit integer additive set-graceful labelings.}
\end{problem}

\begin{problem}{\rm
Verify the existence of integer additive set-graceful labelings for different graph operations and graph products.} 
\end{problem}

\begin{problem}{\rm
Analogous to set-sequential labelings, define integer additive set-sequential labelings of graphs and their properties.}
\end{problem}

\begin{problem}{\rm
Characterise different graph classes which admit integer additive set-sequential labelings.}
\end{problem}

\begin{problem}{\rm
Verify the existence of integer additive set-sequential labelings for different graph operations and graph products.}
\end{problem}

The integer additive set-indexers under which the vertices of a given graph are labeled by different standard sequences of non negative integers, are also worth studying.   All these facts highlight a wide scope for further studies in this area.

\end{document}